\documentclass[10pt]{article}
\usepackage{amsmath,amssymb,amsthm}
\usepackage{graphicx,subfig,float,url,color}
\usepackage{mathrsfs}
\usepackage[colorlinks=true]{hyperref}
\usepackage{pdfsync}
\usepackage{stmaryrd} % \llbracket et \rrbracket

\graphicspath{{fig/}}

\def\N{\textrm{I\kern-0.21emN}}

\newcommand{\R} {\mathbb{R}}

\DeclareMathOperator{\diag}{diag}

\renewcommand{\geq}{\geqslant}
\renewcommand{\leq}{\leqslant}

\newtheorem{proposition}{Proposition}

\newtheorem{definition}{Definition}
\newtheorem{lemma}{Lemma}

\theoremstyle{definition}\newtheorem{remark}{Remark}

\title{Redundancy implies robustness for bang-bang strategies}

\author{Antoine Olivier \thanks{Sorbonne Universit\'es, UPMC Univ Paris 06, CNRS UMR 7598, Laboratoire Jacques-Louis Lions, F-75005, Paris, France.} \thanks{CNES, Direction des lanceurs, 52 rue Jacques Hillairet, 75612 Paris cedex, France.} \footnotemark[4] \; Thomas Haberkorn \thanks{Universit\'e d'Orl\'eans, Laboratoire MAPMO, Rue de Chartres, B. P. 6759 - 45067 Orl\'eans cedex 2, France.} \footnotemark[4] \; Emmanuel Tr\'elat \footnotemark[1] \footnotemark[4]  \\
\'Eric Bourgeois \footnotemark[2] \footnotemark[4] \; David-Alexis Handschuh \footnotemark[2]
\thanks{e-mail: \href{mailto:olivier@ljll.math.upmc.fr}{olivier@ljll.math.upmc.fr} (corresponding author), \href{mailto:emmanuel.trelat@upmc.fr}{emmanuel.trelat@upmc.fr},
\href{mailto:thomas.haberkorn@univ-orleans.fr}{thomas.haberkorn@univ-orleans.fr}, 
\href{mailto:eric.bourgeois@cnes.fr}{eric.bourgeois@cnes.fr},
\href{mailto:david-alexis.handschuh@cnes.fr}{david-alexis.handschuh@cnes.fr}.} }

\date{}

\begin{document}

\maketitle

\begin{abstract}
We develop in this paper a method ensuring robustness properties to bang-bang strategies, for general nonlinear control systems. Our main idea is to add bang arcs in the form of needle-like variations of the control. With such bang-bang controls having additional degrees of freedom, steering the control system to some given target amounts to solving an overdetermined nonlinear shooting problem, what we do by developing a least-square approach. In turn, we design a criterion to measure the quality of robustness of the bang-bang strategy, based on the singular values of the end-point mapping, and which we optimize. Our approach thus shows that redundancy implies robustness, and we show how to achieve some compromises in practice, by applying it to the attitude control of a 3d rigid body.
\end{abstract}

%%%%%%%%%%%%%%%%%%%%%%%%%%%%%%%
% Introduction
%%%%%%%%%%%%%%%%%%%%%%%%%%%%%%%
\section{Introduction}

\subsection{Overview of the method}
\label{overview}
To introduce the subject, we explain our approach on the control problem consisting of steering the finite-dimensional nonlinear control system 
\begin{equation}\label{sys_intro}
\dot x(t) = f(t,x(t),u(t)),
\end{equation}
from a given $x(0)=x_0$ to the target point $x(t_f)=x_f$, with a scalar control $u$ that can only switch between two values, say $0$ and $1$. The general method, as well as all assumptions, will be written in details in a further section.

Let $E(x_0,t_f,u) = x(t_f)$ be the end-point mapping, where $x(\cdot)$ is the solution of \eqref{sys_intro} starting at $x(0)=x_0$ and associated with the control $u$. One aims at finding a bang-bang control $u$, defined on $[0,t_f]$ for some final time $t_f>0$, such that $E(x_0,t_f,u) = x_f$.

Many problems impose to implement only bang-bang controls, i.e., controls saturating the constraints but not taking any intermediate value. These are problems where only external actions of the kind on/off can be applied to the system.

Of course, such bang-bang controls can usually be designed by using optimal control theory (see \cite{lee1967foundations,pontryagin1987,Trelat1}). 
For instance, solving a minimal time control problem, or a minimal $L^1$ norm as in \cite{Caponigro2015}, is in general a good way to design bang-bang control strategies. However, due to their optimality status, such controls often suffer from a lack of robustness with respect to uncertainties, model errors, deviations from the target. Moreover, when the Pontryagin maximum principle yields bang-bang controls, such controls have in general a minimal number of switchings: in dimension $3$ for instance, it is proved in \cite{Krener,Kupka,Schattler} (see also \cite{BonnardChyba,Bonnard2005,Trelat2012} for more details on this issue) that, locally, minimal time trajectories of single-input control-affine systems have generically two switchings. Taking into account the free final time, this makes three degrees of freedom, which is the minimal number to generically make the trajectory reach a target point in $\R^3$, i.e., to solve three (nonlinear) equations.

In these conditions, a natural idea is to add redundancy to such bang-bang strategies, by enforcing the control to switch more times than necessary. These additional switching times are introduced by \emph{needle-like variations}, as in the classical proof of the Pontryagin maximum principle (see \cite{lee1967foundations,pontryagin1987}).

We recall that a needle-like variation $\pi_1=(t_1,\delta t_1,u_1)$ of a given control $u$ is the perturbation $u_{\pi_1}$ of the control $u$ given by
\begin{equation} \label{eq_needle}
u_{\pi_1}(t) = \left\{
\begin{array}{rcl}
u_1  & \textrm{if} & t\in [t_1,t_1+\delta t_1], \\
 u(t) & \textrm{otherwise,}&
\end{array}\right.
\end{equation}
where $t_1\in [0,t_f]$ is the time at which the spike variation is introduced, $\delta t_1$ is a real number of small absolute value that stands for the duration of the variation, and $u_1 \in [0,1]$ is some arbitrary element of the set of values of controls. When $\delta t_1<0$, one replaces the interval $[t_1,t_1+\delta t_1]$ with $[t_1+\delta t_1,t_1]$ in \eqref{eq_needle}.
It is well known that, if $\vert\delta t_1\vert$ is small enough, the control $u_{\pi_1}$ is admissible (that is, the associated trajectory solution of \eqref{sys_intro} is well-defined on $[0,t_f]$) and generates a trajectory $x_{\pi_1}(\cdot)$, which can be viewed as a perturbation of the nominal trajectory $x(\cdot)$ associated with the control $u$, and which steers the control system to the final point
\begin{equation}\label{eq_var}
E(x_0,t_f,u_{\pi_1}) = E(x_0,t_f,u) + \vert \delta t_1 \vert\, v_{\pi_1}(t_f) + o(\delta t_1),
\end{equation}
where the so-called variation vector $v_{\pi_1}(\cdot)$ is the solution of some Cauchy problem related to a linearized system along $x(\cdot)$ (see \cite{lee1967foundations,pontryagin1987,Silva2010} and Proposition \ref{prop_differentiability_epm}).
%\begin{align*}
%\dot{v}_{\pi_1}(t) & = \frac{\partial f}{\partial x}(t, {x}_u(t), {u}(t))v_{\pi_1}(t)\\
%v_{\pi_1}(t_1) & =  f(t_1,{x}_u(t_1),u_1) - f(t_1,{x}_u(t_1), {u}(t_1)) 
%\end{align*}
Recall that the \emph{first Pontryagin cone} $K(t_f)$ is the smallest closed convex cone containing all variation vectors $v_{\pi_1}(t_f)$; it serves as a local convex estimate of the set of reachable points at time $t_f$ (with initial point $x_0$). 

%\medskip

\begin{figure}[h]
\begin{center}
	\includegraphics[scale=1.4]{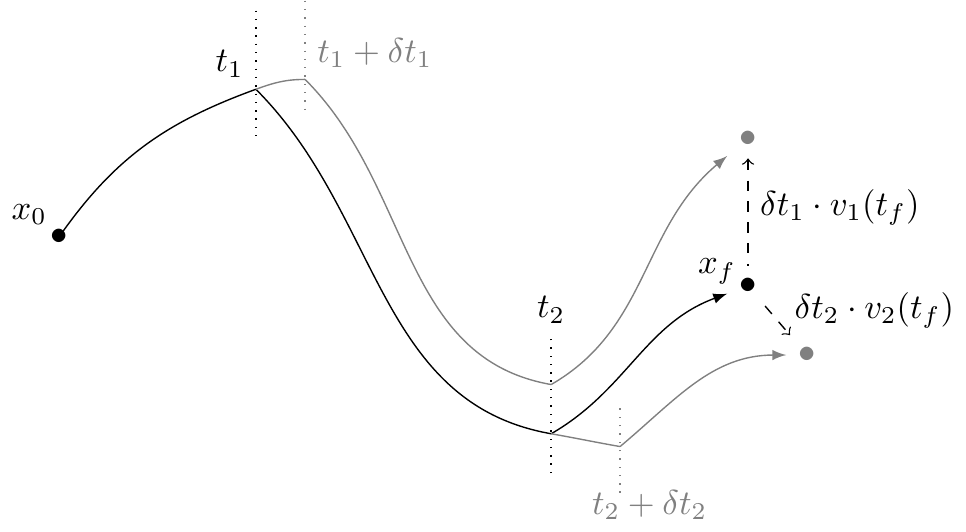}
  	\caption{Changing the switching times induces a displacement at the final time.}
  	\label{fig_diff_epm}
\end{center}
\end{figure}

Assume that the nominal control $u$, which steers the system from $x_0$ to the target point $x_f$, is bang-bang and switches $N$ times between the extreme values $0$ and $1$ over the time interval $[0,t_f]$. We denote by $\mathcal{T} = (t_1, \ldots, t_N)$ the vector consisting of its switching times $0<t_1<\cdots<t_N<t_f$. Then the control $u$ can equivalently be represented by the vector $\mathcal{T}$, provided one makes precise the value of $u(t)$ for $t\in(0,t_1)$. One can also add new switching times: for instance if $u(t)=0$ for $t\in(0,t_1)$, given any $s_1\in(0,t_1)$, the needle-like variation $\pi_1=(s_1,\delta s_1,1)$ (with $\vert\delta s_1\vert$ small enough) is a bang-bang control having two new switching times at $s_1$ and $s_1+\delta s_1$.

In what follows, we designate a bang-bang control either by $u$ or by the set $\mathcal{T}=(t_1, \ldots, t_N)$ of its switching times. This is with a slight abuse because we should also specify the value of $u$ along the first bang arc. But we will be more precise, rigorous and general in a further section.
The end-point mapping is then reduced to the switching times, and one has $E(x_0, t_f, \mathcal{T}) = x_f$.
A variation $\delta\mathcal{T} = (\delta t_1, \ldots, \delta t_N)$ of the switching times generates $N$ variation vectors $(v_{1}(t_f), \ldots, v_{N}(t_f))$, and the corresponding bang-bang trajectory reaches at time $t_f$ the point (see Figure \ref{fig_diff_epm}, where two variations vectors are displayed, for two switching times $t_1$ and $t_2$)
\[
E\left(x_0, t_f, \mathcal{T} + \delta\mathcal{T}\right) = x_f + \delta t_1\cdot v_{1}(t_f) + \cdots + \delta t_N \cdot v_{N}(t_f) + \mathrm{o}(\| \delta\mathcal{T} \|) .
\]
Therefore the end-point mapping $E$ is differentiable with respect to $\mathcal{T}$, and
\begin{equation}\label{eq_diff_epm_intro}
\frac{\partial E}{\partial\mathcal{T}}(x_0, t_f, \mathcal{T})\cdot \delta\mathcal{T} = \delta t_1\cdot v_{1}(t_f) + \cdots + \delta t_N \cdot v_{N}(t_f) .
\end{equation}
Notice that compared to \eqref{eq_var}, the absolute values disappear. We will prove this result in details further in the paper.
In particular, the range of this differential is the first Pontryagin cone $K(t_f)$ (see also \cite{Silva2010}).
Obviously, the more switching times (i.e., degrees of freedom), the more accurate the approximation of the reachable set. 

We now add \emph{redundant} switching times $(s_1, \ldots, s_\ell)$ for some $\ell \in \N$ in order to generate more degrees of freedom to solve the control problem
$$
E\left(x_0,t_f,(t_1, \ldots, t_N, s_1, \ldots, s_\ell)\right) = x_f .
$$
We order the times in the increasing order and we still denote by $\mathcal{T}$ the vector of all switching times.

\paragraph{Redundancy creates robustness.}
We will see further that these redundant switching times contribute to make the trajectory robust to external disturbances or model uncertainties, we will develop a method to tune the switching times in order to absorb these perturbations and steer the system to the desired target $x_f \in \mathbb{R}^n$.

Here, in this still informal introduction, we show how to use the additional switching times to make the system reach targets $x_f+\delta x_f$ in a neighborhood of $x_f$. The idea is to solve the nonlinear system of equations
\begin{equation*}
E(x_0,t_f,\mathcal{T}+\delta\mathcal{T}) = x_f +\delta x_f .
\end{equation*}
Using \eqref{eq_diff_epm_intro}, we propose to solve, at the first order,
\begin{equation}\label{eq_rec}
\frac{\partial E}{\partial\mathcal{T}}(x_0, t_f, \mathcal{T})\cdot \delta\mathcal{T} = \delta x_f ,
\end{equation}
which makes $n$ equations with $N+\ell$ degrees of freedom. 
We assume that $N+\ell$ is (possibly much) larger than $n$ and that the matrix in \eqref{eq_rec} is surjective. Then one can solve \eqref{eq_rec} by using the \emph{Moore-Penrose pseudo-inverse} $\left(\frac{\partial E}{\partial\mathcal{T}}\right)^{\dagger}$ of $\frac{\partial E}{\partial\mathcal{T}}$ (see \cite{Golub}, or see \cite{Beutler1,Beutler2} for a theory in infinite dimension), which yields the solution of minimal Euclidean norm
\begin{equation*}
\delta\mathcal{T} = \left(\frac{\partial E}{\partial\mathcal{T}}\right)^{\dagger} \cdot \delta x_f ,
\end{equation*}
and we have
\begin{equation}\label{eq_estimate_simple}
\left\| \delta\mathcal{T} \right\|_2 \leq \frac{\left\| \delta x_f \right\|_2}{\sigma_{min}} ,
\end{equation}
where $\sigma_{min}$ is the smallest positive singular value of $\frac{\partial E}{\partial\mathcal{T}}$. This estimate gives a natural measure for robustness, that we will generalize.

\medskip

The two main contributions of this paper are:
\begin{itemize}
\item the idea of adding redundant switching times in order to make a nominal bang-bang control more robust, while keeping it as being bang-bang;
\item the design of a practical tracking algorithm, consisting of solving an overdetermined nonlinear system by least-squares, thus identifying a robustness criterion that we optimize.
\end{itemize}
They are developed in a rigorous and general context in the core of the paper.

\subsection{State of the art on robust control design}
There is an immense literature on robust control theory, with many existing methods in order to efficiently control a system subjected to uncertainties and disturbances. Whereas there are many papers on $\mathcal{H}_2$ and $\mathcal{H}_{\infty}$ methods, except a few contributions in specific contexts, we are not aware of any general theory allowing one to tackle perturbations by using only bang-bang controls. This is the focus of this paper. 

Let us however shortly report on robustness methods when one is not bound to design bang-bang controls.
In \cite{Koh:1999}, a path-tracking algorithm with bang-bang controls is studied, for a double integrator and a wheeled robot. The technique relies heavily on the expression of the equations and does not apply to more general systems. In \cite{singh1994}, the authors build a robust minimal time control for spacecraft's attitude maneuvers by canceling the poles of some transfer function. A remarkable fact is that the robustified control presents more switchings than the minimal time control. In this case, the robustness is evaluated as the maximum amplitude on a Bode diagram (see also \cite{liu1992} and \cite{liu1993} for similar works). In \cite{You2000}, the authors observe that bang-bang controls are intrinsically not robust, and use pieces of singular trajectories (hence, not bang-bang) to overcome this issue.

In the $\mathcal{H}_2$ and $\mathcal{H}_{\infty}$ theories, control systems are often written in the frequency domain using the Laplace transform. For a transfer matrix $G(s)$, the two classical measures for performance are (see \cite{Doyle1989,zhou1996robust}) the $\mathcal{H}_2$ norm and the $\mathcal{H}_{\infty}$ norm respectively:
\begin{equation*}
\left\|G\right\|_{2} = \left(\frac{1}{2 \pi} \int^{+ \infty}_{- \infty}{\text{Trace}(G(j \omega) G(j \omega)^*) d\omega} \right)^{1/2}\qquad\textrm{and}\qquad
\left\|G\right\|_{\infty} = \sup_{\omega \in \R} \overline{\sigma}(G(j \omega)),
\end{equation*}
where $\overline{\sigma}(G)$ is the largest singular value of $G$.

In the linear quadratic theory, the question of optimal tracking has been widely addressed: given a reference trajectory $\xi(\cdot)$, we track it with a solution of some control system $\dot{x}(t) =f(x(t),u(t))$, minimizing a cost of the form 
\begin{equation*}
\int_0^{t_f} \left( \Vert x(t)-\xi(t)\Vert_W^2 + \Vert u(t)\Vert_U^2 \right)\, dt +  \Vert x(t_f)-\xi(t_f)\Vert_Q^2 ,
\end{equation*}
with weighted norms (see \cite{Anderson,Kwakernaak,Trelat1}).
The first term in the integral measures how close one is to the reference trajectory, the second one measures a $L^2$ norm of the control (energy), and the third one accounts for the distance at final time between the reference trajectory $\xi(\cdot)$ and $x(\cdot)$. Then, the control can be expressed as a feedback function of the error $x(t)-\xi(t)$, involving the solution of some Riccati equation. In \cite{AndreaNovel2013,Khalil}, the authors investigate the question of stabilizing around a slowly time-varying trajectory. They also introduce uncertainties on the model and study the sensitivity of the system to those uncertainties. In the case of the existence of a delay on the input, a feedback law is proposed. In \cite{Lin2007,Tan2009}, uncertainties $p$ are introduced in a linear system $\dot{x}(t) = A(p)x(t) + B u(t)$, and a tracking algorithm is suggested, under matching conditions on the uncertainties or not (see also \cite{Abdallah1991} for a survey on robust control for rigid robots).

In the late 1970's, $\mathcal{H}_{\infty}$ control theory developed. The control system is often described by a plant $G$ and a controller $K$. Then, the dependency of the error $z$ (to be minimized) on the input $v$ can be written as $z = F(G,K) v$.
The $\mathcal{H}_{\infty}$ control problem consists of finding the best controller $K$ such that the $\mathcal{H}_{\infty}$ norm of the matrix $F(G,K)$ is minimized: $\left\|F(G,K) \right\|_{\infty} = \sup_{\omega \in \R} \overline{\sigma}(F(G,K)(j \omega))$. It can be interpreted as the maximum gain from the input $v$ to the output $z$.
This criterion was introduced in order to deal with uncertainties on the model (on the plant $G$). In \cite{Zames1981}, the author introduced the notion and highlighted the connection with robustness. In \cite{Doyle1989}, a link is shown between the existence of such a controller and conditions on the solutions of two Riccati equations. Following a notion introduced in \cite{Gahinet1992}, the linear matrix inequality (LMI) approach was introduced in \cite{Gahinet1994}, and used in \cite{apkarian2004,noll2006} to solve the $\mathcal{H}_{\infty}$ synthesis. The Riccati equations are replaced with Riccati inequalities, whose set of solutions parameterizes the $\mathcal{H}_{\infty}$ controllers (see also \cite{boyd1994} for the use of LMIs in control theory). The papers \cite{Doyle1981,McFarlane1992,Xie1992} present design procedures in this context to elaborate the feedback controller $K$. In \cite{Ge1996}, the theory is extended to systems with parameters uncertainties and state delays, as well as in \cite{Xu2006}, with stochastic uncertainty.

In many optimal control problems, the application of the Pontryagin maximum principle leads to bang-bang control strategies, and the classical $\mathcal{H}_{2}$ and $\mathcal{H}_{\infty}$ theories were not designed for such a purpose. But the optimal trajectories are in general not robust. Adding needle-like variations is therefore a way to improve robustness, and is the main motivation of this paper. Of course, the method applies to any bang-bang control strategy, not necessarily optimal.

The approach that we suggest in this paper combines an off-line treatment of the control strategies, with a feedback algorithm based on the structure of the control. We emphasize here that this algorithm preserves the bang-bang structure of the control. It consists of applying a nominal control strategy (that needs to be computed \emph{a priori}), and adjusting it in real time, allowing one to track a nominal trajectory. The off-line method takes a solution of the control problem and makes it more robust by adding additional switching times (i.e.,  \emph{redundancy}), which can be seen as additional degrees of freedom. Note that our analysis is done in the state space, without needing to consider the frequency domain. A key ingredient to the method is the use of needle-like variations. 

\subsection{Structure of the paper}

The paper is organized as follows. In Section \ref{sec2}, we develop an algorithm to steer a perturbed system to the desired final point. The method is similar to the one presented in Section \ref{overview}, except that we need to consider a backward problem. Indeed, the final point is fixed, and perturbations appear all along the trajectory. Besides, our measure for robustness comes out naturally in view of \eqref{eq_estimate_simple}. Having identified the robustness criterion, we show in Section \ref{sec3} how to add redundant switching times, leading one to solve a finite-dimensional nonlinear optimization problem. In Section \ref{sec4}, we provide some numerical illustrations on the attitude control problem of a 3-dimensional rigid body.

%%%%%%%%%%%%%%%%%%%%%%%%%%%%
% Section tracking
%%%%%%%%%%%%%%%%%%%%%%%%%%%%
\section{Tracking algorithm}
\label{sec2}
\paragraph{Setting.} 
In this paper, we consider the control system 
\begin{equation}
\dot{x}(t) = f(t,x(t),u(t)),
\label{dynamics}
\end{equation}
where $f$ is a smooth function $\R \times \R^n \times \R^m \rightarrow \R^n$, the state $x(\cdot) \in \R^n$, the control $u(\cdot) \in L^{\infty}([0,t_f];\Omega)$, and $\Omega$ is the subset of $\R^m$: $[a_1,b_1] \times \cdots \times [a_m,b_m]$. We make two additional hypothesis: the controls we consider are ``bang-bang'', with a finite number of switching times:
\begin{center}
\begin{tabular}{cc}
$(H_1)$ & $\forall i \in \llbracket1,m \rrbracket$, $u_i(t) \in \left\{a_i,b_i\right\}$, a.e. \\
$(H_2)$ & $\forall i \in \llbracket1,m \rrbracket$, $u_i$ does not chatter. \\
\end{tabular}
\end{center}
A control is chattering when it switches infinitely many times over a compact time interval (see \cite{zhu2016,fuller1963}). Therefore, our method does not apply to those controls. However, when the solution of an optimal control problem chatters, provided that it is possible, one could consider a sub-optimal solution, with only a finite number of switching times.

In the context of optimal control, we will denote the cost under the form 
\begin{equation}
C(u) = \int^{t_f}_0{f^0(t,x(t),u(t))\,dt}.
\label{cost}
\end{equation}

We recalled in the introduction the (classical) definitions of the end-point mapping, of a needle-like variation (\ref{eq_needle}) and the expansion of the end-point mapping subject to a needle-like variation (\ref{eq_var}).

\subsection{Reduced end-point mapping}
\label{reduced_epm}

In this subsection, we give the definition of the reduced end-point mapping, and show a differentiability property.

Let us consider a bang-bang control $u(\cdot)$, and its associated trajectory $x(\cdot)$. For the sake of simplicity, we make the additional assumption that for every switching time $t_j$, one and only one component of the control commutes. Therefore, provided we specify the initial value of each component, the control $u$ is entirely characterized by the switching times of its components and can be represented by a vector:
\begin{equation*}
\left((u_{10}, \ldots,u_{m0}),\left(t_1,i_1\right), \ldots, \left(t_N,i_N\right), t_f \right) \in \Omega \times \R^{2N+1},
\end{equation*}
where $u_{i0} \in \{a_i,b_i\}$ is the initial value for the control $u_i(\cdot)$ ($i \in \llbracket1,m\rrbracket)$, $N$ is the total number of switching times, $t_f$ is the final time, and $i_j$ is the component of the control that switches at time $t_j$. As this representation entirely characterizes the control, we will use indistinctly the notation $u$ and $\left((u_{10}, \ldots,u_{m0}),\left(t_1,i_1\right), \ldots, \left(t_N,i_N\right), t_f \right)$ to speak about the control whose components switch at the times $t_j$. In the literature, $\left(\left(t_1,i_1\right), \ldots, \left(t_N,i_N\right)\right)$ is often called a switching sequence.

\begin{remark}
Had we wanted to allow simultaneous switching of multiple components, we would need to consider controls represented by:
\begin{equation*}
\left((u_{10}, \ldots,u_{m0}),\left(t_1,\mathcal{I}_1\right), \ldots, \left(t_N,\mathcal{I}_N\right), t_f \right),
\end{equation*}
where $\mathcal{I}_j \subset \llbracket1,m \rrbracket$ represents the set of components that switch at time $t_j$.
\end{remark}

\begin{definition}[Reduced end-point mapping]
We define the \emph{reduced end-point mapping} by 
\begin{equation*}
E(x_0,(u_{10}, \ldots,u_{m0}),\left(t_1,i_1\right), \ldots, \left(t_N,i_N\right),t_f) = x_u(x_0,t_f),
\end{equation*}
where $u$ is the control represented by $\left((u_{10}, \ldots,u_{m0}),\left(t_1,i_1\right), \ldots, \left(t_N,i_N\right), t_f \right)$, and $x_u(x_0,t_f)$ is the associated state at time $t_f$, starting at $x_0$.
\end{definition}
Note that in \cite{Maurer2005,Maurer2004}, the authors also reduce a bang-bang control to its switching points, in order to formulate an optimization problem in finite-dimension.

 In the following, when writing this reduced end-point mapping, we may consider that the initial point $x_0$ is fixed, as well as the way the components of the control switch (i.e., we consider that the N-tuple $(i_1, \ldots, i_N)$ is fixed), the initial values $u_{i0}$ and the final time $t_f$. In this context, we may forget them in the notations, and denote the reduced end-point mapping by
\begin{equation*}
E(t_1, \ldots, t_N) = x_u(t_f).
\end{equation*}

% Differentiabilite application entree sortie reduite
A remarkable fact is that the reduced end-point-mapping is differentiable. Compared to the expansion \eqref{eq_var} with respect to a needle-like variation, the sign of $\delta t$ does not matter. For the sake of completeness, we give the proof in appendix.
\begin{proposition}
\label{prop_differentiability_epm}
The reduced end-point mapping is differentiable, and 
\begin{equation*}
dE(t_1 , \ldots, t_N) = 
\begin{pmatrix}
v_1(t_f) & \cdots & v_N(t_f)
\end{pmatrix}
\in \mathcal{M}_{n,N}(\R),
\end{equation*}
where $v_j(\cdot)$ ($j \in \llbracket1,N \rrbracket$) is the solution of the Cauchy problem, defined for $t \geq t_j$:
\begin{align*}
\dot{v}_j(t) &= \frac{\partial f}{\partial x}(t,x(t),u(t))v_j(t) \\
v_j(t_j) &= \left\{
\begin{tabular}{rl}
$f(t_j,x(t_j),(\ldots, a_{i_j}, \ldots)) - f(t_j,x(t_j),u(t_j^+))$ & if $u_{i_j}$ switches from $a_{i_j}$ to $b_{i_j}$. \\
$f(t_j,x(t_j),(\ldots, b_{i_j}, \ldots)) - f(t_j,x(t_j),u(t_j^+))$ &  if $u_{i_j}$ switches from $b_{i_j}$ to $a_{i_j}$.
\end{tabular}
\right.
\end{align*}
The notation $(\ldots, a_{i_j}, \ldots)$ (resp. $(\ldots, b_{i_j}, \ldots)$) is used to show a difference with $u(t_j^+)$ (resp. $u(t_j^-)$) on the $i_j$-th component only.
\end{proposition}

\begin{remark}
\label{rq_affine}
In the special (and important in practice) case of a control-affine system 
\[
\dot{x}(t) = f_0(x(t)) + \sum_{k=1}^m{u_k(t) f_k(x(t))},
\]
the initial condition on $v_j$ can be written much more easily:
\begin{equation*}
v_j(t_j) = \left\{
\begin{tabular}{rl}
$ (a_{i_j} - b_{i_j}) f_{i_j}(x(t_j))$ & if $u_{i_j}$ switches from $a_{i_j}$ to $b_{i_j}$. \\
$(b_{i_j} - a_{i_j})  f_{i_j}(x(t_j))$ & if $u_{i_j}$ switches from $b_{i_j}$ to $a_{i_j}$.
\end{tabular}
\right.
\end{equation*}
\end{remark}

\subsection{Absorbing perturbations}
\label{backward_epm}

As explained in the introduction, we present in this paper a closed-loop method to actually steer the system towards a point $x_f$, with bang-bang controls, even in the presence of perturbations. 

First, for the sake of simplicity, we will explain how to control the system to some point $x_f + \delta x_f$. We will see that this idea can be adapted for our purpose of controlling a perturbed trajectory, by simply reversing the time.

\paragraph{Perturbations on the final point.}

We briefly generalize the problem introduced in the introduction. Let 
\[ 
\overline{u} = \left((u_{10}, \ldots,u_{m0}),\left(\overline{t}_1,i_1\right), \ldots, \left(\overline{t}_N,i_N\right), t_f \right) \in \Omega \times \R^{2N+1}
\] 
be a control such that $x_{\overline{u}}(t_f) = x_f$. That is, using the definition of Subsection \ref{reduced_epm}, we have that 
\begin{equation*}
E(x_0,(u_{10}, \ldots,u_{m0}),\left(\overline{t}_1,i_1\right), \ldots, \left(\overline{t}_N,i_N\right), t_f) = x_f.
\end{equation*}
Or, considering that the final time $t_f$, the initial point $x_0$, the components $(i_1, \ldots,i_N)$ and the initial values $(u_{10}, \ldots,u_{m0})$ are fixed, 
\begin{equation*}
E(\overline{t}_1, \ldots, \overline{t}_N) = x_f.
\end{equation*}
Let $\delta x_f$ be some perturbation of the final point $x_f$. We look for a vector $\delta \mathcal{T} = (\delta t_1, \ldots, \delta t_N)$ so that the system reaches the target point $x_f + \delta x_f$:
\begin{equation*}
E(\overline{t}_1 + \delta t_1, \ldots, \overline{t}_N + \delta t_N) = x_f + \delta x_f.
\end{equation*}
As we have shown in Proposition \ref{prop_differentiability_epm} the differentiability of the reduced end-point mapping, we can write 
\begin{equation*}
E(\overline{t}_1 + \delta t_1, \ldots, \overline{t}_N + \delta t_N) = E(\overline{t}_1, \ldots,\overline{t}_N) + dE(\overline{t}_1, \ldots,\overline{t}_N) \cdot \delta \mathcal{T} + o(\| \delta \mathcal{T}\|).
\end{equation*}
At order one, the solution is given by the solution of the linear equation
\begin{equation*}
dE(\overline{t}_1, \ldots,\overline{t}_N) \cdot \delta \mathcal{T} = \delta x_f.
\end{equation*}
It is natural to target the final point $x_f +\delta x_f$ while shifting the switching times as little as possible. That is, we look for the solution of minimal euclidian norm of the previous equation, which is given by $\delta \mathcal{T} = dE(\overline{t}_1, \ldots, \overline{t}_N)^{\dagger} \cdot \delta x_f$.

Therefore, we have shown how to compute, at order one, the correction to apply to control the system to some point $x_f + \delta x_f$: it boils down to solving a least-squares problem. Let us keep in mind that our definitive goal is to control systems that are perturbed all along their trajectory, to a fixed final point $x_f$. In other words, from a perturbed point $x(t) + \delta x(t)$ at some time $t \in [0,t_f)$, we want to absorb the perturbation $\delta x(t)$ and still reach the final point $x_f$. Even if this is a slightly different setting, we show that we can apply the same idea if we look at a \emph{backward problem}.

\paragraph{Absorbing a perturbation at time $t$.}
Let $(\overline{x}(\cdot),\overline{u}(\cdot))$ be a nominal solution of the control system \eqref{dynamics}. We assume that when applying in practice the control $\overline{u} = \overline{\mathcal{T}}$, because of model uncertainties and perturbations, we observe a perturbed trajectory $x_{per}(t) = \overline{x}(t) + \delta x(t)$.

Let $t \in [0,t_f]$. Starting from the perturbed point $\overline{x}(t) + \delta x(t)$, which stands as a new initial point, we want to reach the final point $x_f$ in time $t_f-t$. Hence, we look for a control $\overline{u} + \delta u$ such that 
\[
E(\overline{x}(t) + \delta x(t),\overline{u}+\delta u,t_f-t) = x_f.
\]
Assume for a moment that the perturbation of the control $\delta u$ is small in $L^{\infty}$ norm. Then, at least formally, one can write
\[
E(\overline{x}(t),\overline{u},t_f-t) + \frac{\partial E}{\partial x_0}(\overline{x}(t),\overline{u},t_f-t)\cdot \delta x(t) + \frac{\partial E}{\partial u}(\overline{x}(t),\overline{u},t_f-t)\cdot \delta u  + o(\|\delta x(t)\|+\|\delta u\|)= x_f.
\]
Therefore, at order one, we look for a solution of the (linear) equation
\begin{equation}
\frac{\partial E}{\partial x_0}(\overline{x}(t),\overline{u},t_f-t)\cdot \delta x(t) + \frac{\partial E}{\partial u}(\overline{x}(t),\overline{u},t_f-t)\cdot \delta u  = 0.
\label{eq_perturbations}
\end{equation}

However, we do not want, in this paper, to apply small perturbations in the $L^{\infty}$ norm, as they would not result in bang-bang controls (However, this is similar to what is done while performing a Ricatti procedure to stabilize a system or track a reference trajectory). Nevertheless, reducing the end-point mapping to the switching times enables us to preserve the bang-bang structure: in the formalism previously introduced, we need to solve the nonlinear system of equations 
\begin{equation*}
E(\overline{x}(t) + \delta x(t), \overline{\mathcal{T}} + \delta \mathcal{T}, t_f - t) = x_f.
\end{equation*}

The equation \eqref{eq_perturbations} becomes
\begin{equation}
\frac{\partial E}{\partial \mathcal{T}}(\overline{x}(t),\overline{\mathcal{T}},t_f-t)\cdot \delta \mathcal{T} = - \frac{\partial E}{\partial x_0}(\overline{x}(t),\overline{\mathcal{T}},t_f-t)\cdot \delta x(t),
\label{eq_perturbations_1}
\end{equation}
where the expression $\partial E / \partial \mathcal{T}$ is given by Proposition \ref{prop_differentiability_epm}. 

\paragraph{A backward problem.} Solving this equation requires the computation of the partial differential $\partial E/\partial x_0$ at the initial point $\bar{x}(t)$. We will see now that it can be overcome by introducing a backward problem. Of course, the two formulations are equivalent.

\begin{definition}[Backward end-point mapping]
Let $u = (t_1, \ldots, t_N)$ be a bang-bang control, and $t \in [0,t_f]$. We define the backward end-point mapping by
\begin{equation*}
\tilde{E}(t, t_1, \ldots, t_N) = \tilde{x}(t_f - t),
\end{equation*}
where $\tilde{x}(\cdot)$ is the solution to the Cauchy problem
\begin{align*}
\dot{\tilde{x}}(t) &= -f(t_f - t, \tilde{x}(t),u(t_f - t)), \\
\tilde{x}(0) &= x_f.
\end{align*}
\end{definition}

Note that for the nominal trajectory $(\overline{x}(\cdot),\overline{u}(\cdot))$, we have that 
\begin{equation*}
\tilde{E}(t, \overline{t}_1, \ldots, \overline{t}_N) = \overline{x}(t).
\end{equation*}
Indeed, we have in this case that $\overline{x}(t) = \tilde{x}(t_f - t)$: if we integrate the nominal system backward, starting from the point $x_f$ during a time period $t_f - t$, we end up at point $\overline{x}(t)$.

\begin{remark}
Let $t \in [0,t_f]$, and $j$ be the smallest index such that $\overline{t}_j > t$ (with the convention that $j = N+1$ if $t>t_N$). Then, note that $\overline{t}_1, \ldots, \overline{t}_{j-1}$ do not play any role in the computation of $\tilde{E}(t, \overline{t}_1, \ldots, \overline{t}_N)$. The differential of $\tilde{E}$ can be computed with the Proposition \ref{prop_differentiability_epm}. It is a matrix of size $n \times (N - j + 1)$. 
\end{remark}

In this context, the problem of adjusting the system back towards $x_f$ writes: at time $t$, find $(t_j, \ldots, t_N)$ such that 
\begin{equation}
\tilde{E}(t,{t}_1, \ldots, {t}_N) = x_{per}(t).
\label{prob_backward}
\end{equation}
  
We see that reversing the time, we place ourselves in the setting previously described of aiming at a perturbed final point. Therefore, we have the following proposition.
  
\begin{proposition}
At order one in $\delta x$, the solution of minimal norm of the problem \eqref{prob_backward} is given by $\overline{\mathcal{T}} + \delta \mathcal{T}$, with
\begin{equation}
\delta \mathcal{T} = d\tilde{E}(t,\overline{\mathcal{T}})^{\dagger}\cdot \delta x(t),
\label{eq_recalage}
\end{equation}
where $d\tilde{E}(t,\overline{\mathcal{T}})^{\dagger}$ denotes the pseudo-inverse of $d\tilde{E}(t,\overline{\mathcal{T}})$. Moreover, we have the estimate
\begin{equation}
\left\| \delta \mathcal{T} \right\|_2 \leq \frac{1}{\sigma_{min}(t)} \left\| \delta x(t) \right\|_2,
\label{estimate_deltat}
\end{equation}
where $\sigma_{min}(t)$ is the smallest positive singular value of $d\tilde{E}(t,\overline{\mathcal{T}})$.
\end{proposition}

\begin{proof}
The scheme of the proof has already been exposed previously in the paper. However, we write it extensively here. Let $\delta \mathcal{T} = \mathcal{T} - \overline{\mathcal{T}}$. The problem writes
\begin{equation*}
\tilde{E}(t,\overline{\mathcal{T}} + \delta \mathcal{T}) = x_{per}(t).
\end{equation*}
According to Proposition \ref{prop_differentiability_epm}, the backward end-point mapping is differentiable (and we also know how to compute its derivative), so
\begin{align*}
\tilde{E}(t,\overline{\mathcal{T}} + \delta \mathcal{T}) 	&= \tilde{E}(t,\overline{\mathcal{T}}) + d\tilde{E}(t,\overline{\mathcal{T}})\cdot \delta \mathcal{T} + o(\| \delta \mathcal{T} \|) \\
							&= \overline{x}(t) + d\tilde{E}(t,\overline{\mathcal{T}})\cdot \delta \mathcal{T} + o(\| \delta \mathcal{T}\|).
\end{align*}
So, at order one, the problem writes 
\begin{equation}
d\tilde{E}(t,\overline{\mathcal{T}})\cdot \delta \mathcal{T} = \delta x(t).
\label{eq_perturbations_2}
\end{equation}
It is well known (see \cite{allaire2002algebre} for instance), that the solution of minimal norm of this equation is $\delta \mathcal{T} = d\tilde{E}(t,\overline{\mathcal{T}})^{\dagger}\cdot \delta x(t)$. Besides, let $\sigma_{max}(t) > \cdots > \sigma_{min}(t) > 0$ denote the positive singular values of $d\tilde{E}(t,\overline{\mathcal{T}})$. We have that $\left\| d\tilde{E}(t,\overline{\mathcal{T}})^{\dagger} \right\|_2 = {1}/{\sigma_{min}(t)}$ ($\left\| \cdot \right\|_2$ for a matrix denotes the induced norm corresponding to the euclidean norm), so that 
\begin{align*}
\left\| \delta \mathcal{T} \right\|_2&= \left\| d\tilde{E}(t,\overline{\mathcal{T}})^{\dagger}\cdot \delta x (t)\right\|_2  \\
					& \leq \left\| d\tilde{E}(t,\overline{\mathcal{T}})^{\dagger} \right\|_2 \cdot \left\| \delta x (t)\right\|_2 \\
					& \leq \frac{\left\| \delta x (t)\right\|_2}{\sigma_{min}},
\end{align*}
which concludes the proof.
\end{proof}

\begin{remark}
We have the relation that, for all vector of switching times $\mathcal{T}$
\begin{equation*}
E(\tilde{E}(t,\mathcal{T}),\mathcal{T},t_f-t) = x_f.
\end{equation*}
Differentiating this equality with respect to $\mathcal{T}$, we have that, for all $\delta \mathcal{T}$
\begin{equation*}
\frac{\partial E}{\partial x_0}(\tilde{E}(t,\mathcal{T}),\mathcal{T},t_f-t)\cdot d\tilde{E}(t,\overline{T})\cdot \delta \mathcal{T} + \frac{\partial E}{\partial \mathcal{T}}(\tilde{E}(t,\mathcal{T}),\mathcal{T},t_f-t)\cdot \delta \mathcal{T} = 0.
\end{equation*} 
Replacing the second term by its value in \eqref{eq_perturbations_1}, it follows that 
\begin{equation*}
\frac{\partial E}{\partial x_0}(\tilde{E}(t,\mathcal{T}),\mathcal{T},t_f-t)\cdot d\tilde{E}(t,\mathcal{T})\cdot \delta \mathcal{T} = \frac{\partial E}{\partial x_0}(\tilde{E}(t,\mathcal{T}),\mathcal{T},t_f-t)\cdot \delta x(t).
\end{equation*}

It is easy to show that $\partial E/\partial x_0$ can be expressed as the resolvent of a linearized system. Therefore, the matrix $\partial E/\partial x_0 $ is invertible, and the equations \eqref{eq_perturbations_1} and \eqref{eq_perturbations_2} are equivalent. But solving \eqref{eq_perturbations_2} only requires to compute the derivative of $\tilde{E}$. This is what we do in the following.
\end{remark}

\begin{remark}
\label{rq_leastsquares}
Note that it might not always be possible to find a solution to the equation $d\tilde{E}(t,\overline{\mathcal{T}})\cdot \delta \mathcal{T} = \delta x(t)$. This may happen for instance if $t > t_{N-n+1}$, i.e., we do not have enough degrees of freedom left to absorb the perturbation $\delta x(t) \in \R^n$. However, we can still give a meaning to the equation $d\tilde{E} \cdot \delta \mathcal{T} = \delta x(t)$. We look for a solution of the least-square problem:
\begin{equation*}
\min_{\delta \mathcal{T} \in \R^N} \left\| d\tilde{E}(t,\overline{t}_1, \ldots, \overline{t}_N) \cdot \delta \mathcal{T} - \delta x (t)\right\|^2_2,
\end{equation*}
for which $\delta \mathcal{T} =  d\tilde{E}(t,\overline{t}_1, \ldots, \overline{t}_N)^{\dagger} \cdot \delta x(t)$ is still the solution of minimal norm (see \cite{allaire2002algebre}). We see here emerging the idea that the number of switching times (i.e., degree of freedom) left at time $t$, is going to be an important factor to track the system back towards the final point $x_f$.
\end{remark}

\paragraph{Numerical algorithm.} 

At time $t$, Equation \eqref{eq_recalage} provides us with a formula to adjust the control so that the perturbed trajectory eventually reaches $x_f$. But it certainly does not enable us to face perturbations that would happen after time $t$. In order to absorb perturbations all along the trajectory, we suggest the following algorithm: Let $\mathcal{T}$ be an initial control. Given an integer $s$ and a subdivision $0 < \tau_1 < \cdots < \tau_s < t_f$ of the interval $[0,t_f]$, we adjust the control at each $\tau_i$ for all $i \in \llbracket 1,s\rrbracket$. That is, for each $i \in \llbracket 1,s\rrbracket$, we measure the drift $\delta x (\tau_i)= x_{per}(\tau_i) - x_{ref}(\tau_i)$, and compute the differential of the backward end-point mapping $d\tilde{E}(\tau_i,\overline{t}_1, \ldots, \overline{t}_N)$. We deduce from \eqref{eq_recalage} that the correction to apply is then $\delta \mathcal{T} = d\tilde{E}(\tau_i,\overline{t}_1, \ldots, \overline{t}_N)^{\dagger} \cdot \delta x(\tau_i)$. We then update the control by considering the new vector of switching times $\mathcal{T} + \delta \mathcal{T}$.

%\begin{algorithm}
%\caption{Tracking algorithm to absorb perturbations}
%\begin{algorithmic}[1]
%\State Choose an integer $s$ and a subdivision $(\tau_1,\ldots,\tau_s)$.
%\State Set $t = 0$.
%% \State Set $x_{ref,0}  = x_0$
%\Comment{Initial conditions}
%\State $\overline{\mathcal{T}}$ 
%\Comment{Initial switching times}
%\For{$i = 0$, $i < N$, $i=i+1$}
%    	\State Integrate the ideal system $f$ from $t$ to $\tau_i$, with initial conditions $x_{ref,0}$.
%    	\State Measure the drift $\delta x (\tau_i)= x_{per}(\tau_i) - x_{ref}(\tau_i)$.
%    	\State Compute the differential of the backward end-point mapping $d\tilde{E}(\tau_i,\overline{t}_1, \ldots, \overline{t}_N)$.
%    	\State Compute the correction $\delta \mathcal{T} = d\tilde{E}(\tau_i,\overline{t}_1, \ldots, \overline{t}_N)^{\dagger} \cdot \delta x(\tau_i)$.
%    	\State Apply the correction $\overline{\mathcal{T}} \gets \overline{\mathcal{T}} + \delta \mathcal{T}$.
%    	\If{$\exists j$ s.t. $\overline{t}_{j+1} < \overline{t}_j$}
%    		\State ``Stop''. Interchanging of switching times.
%    	\EndIf
%    	\State $x_{ref,0} \gets x_{per}(\tau_i)$.
%    	\State $t \gets \tau_i$.
%\EndFor
%\end{algorithmic}
%\label{algo_tracking}
%\end{algorithm} 
%\vspace{0.5cm}

\begin{remark}
\label{rq_interchanging}
When computing the correction $\mathcal{T} + \delta \mathcal{T}$, it may happen that the new switching times are not ordered, i.e.,  there exists some integer $ j \in \llbracket1,N-1\rrbracket$ such that $t_{j+1} < t_j$. In this case, we consider that the correction is not physically acceptable, and we reject it. (Note that in some cases, we may want to continue the integration of the system even if two switching times are not ordered. In that case, we can always use the last admissible control, where all the switching times are ordered.)
\end{remark}

\begin{remark}
The computation of the differential $d\tilde{E}(t,\overline{t}_1, \ldots, \overline{t}_N)$ is done via the integration of a system of ordinary differential equations, which can be done efficiently and quickly using numerical integrators. However, the size of the system (as well as the time required to compute the pseudo-inverse) directly depends on the number of switching times $N$ and on the state dimension $n$.
\end{remark}

%%%%%%%%%%%%%%%%%%%%
% Section robustesse
%%%%%%%%%%%%%%%%%%%%

\section{Promoting robustness}
\label{sec3}

Intuitively, we want to say that a control is robust whenever the correction $\delta \mathcal{T}$ required to absorb the perturbation $\delta x(t)$ is small. Since we have shown the estimate $\left\| \delta \mathcal{T} \right\|_2 \leq \left\| \delta x(t) \right\|_2 / \sigma_{min}(t)$, a robust trajectory is then one for which the values of $1/\sigma_{min}(t)$ remain small along the trajectory.

\begin{definition}
We define the following cost, that we will use to characterize the robustness of a trajectory 
\begin{equation}
C_r(t_1, \ldots, t_N) = \int^{t_N}_0{\frac{1}{\sigma_{min}(t)^2}\, dt}.
\label{cout_robu}
\end{equation}
\end{definition}

\begin{remark}
In the previous definition, the upper bound in the integral is $t_N$, because for $t > t_N$, the backward end-point mapping derivative $d\tilde{E}(t,t_1, \ldots,t_N)$ is not defined, and neither is $\sigma_{min}(t)$. For some reason, we may only want to have robustness up until some time $t^{\star} < t_N$. Then the previous definition would become $\int^{t^{\star}}_0{{1}/{\sigma_{min}(t)^2}dt}$.
\end{remark}

In this section, we show how the switching times of a trajectory can be chosen to build one that is more robust. We also suggest a new way to design a trajectory, by adding redundant switching times, that give us more degrees of freedom. Note also that we will start from a solution of an \emph{optimal} control problem, because it is of high importance in practice, but the method generally applies when starting from any control, as long as it satisfies the hypothesis $(H_1)$ and $(H_2)$. Starting from an initial control such that $E(t_1, \ldots, t_N) = x_f$, we look for \emph{redundant} switching times $(s_1, \ldots,s_l)$ such that $E(t_1, \ldots, t_N,s_1, \ldots,s_l) = x_f$, while minimizing the cost \eqref{cout_robu} that accounts for robustness:
\[
C_r(t_1, \ldots, t_N,s_1, \ldots,s_l).
\]

\subsection{An auxiliary optimization problem}

Let us consider a bang-bang trajectory (satisfying the hypothesis ($H_1$) and ($H_2$)) of the control system \eqref{dynamics}, optimal for the cost \eqref{cost}. That is, $\overline{u} = ((u_{10}, \ldots,u_{m0}),\left(\overline{t}_1,i_1\right), \ldots, \left(\overline{t}_N,i_N\right),t_f)$ is an optimal solution of the optimization problem 
\begin{equation}
\begin{array}{ccc}
\min_{(i_1, \ldots, i_N)} & \min_{(t_1, \ldots,t_N)} & C(t_1, \ldots,t_N). \\
 & \text{s.t.  } E(t_1, \ldots, t_N) = x_f \\
\end{array}
\label{prob_ini}
\end{equation}
Let us emphasize the fact that reducing the control to its switching times enables us to reduce a problem in infinite dimension 
\begin{equation*}
\begin{array}{cc}
\min_{u \in L^{\infty}([0,t_f];\Omega)} & C(u) \\
 \text{s.t.  } E(u) = x_f \\
\end{array}
\end{equation*}
to a finite number of non-linear problems under non-linear constraints in finite dimension, provided we set $N$, as we left aside chattering trajectories.

In order to make the control more robust we suggest to solve the following problem. We fix the components of the control $(i_1, \ldots, i_N)$, and we introduce the cost that accounts for the robustness of a trajectory:
\begin{equation*}
\begin{array}{cc}
\min_{(t_1, \ldots,t_N)} & \lambda_1 C(t_1, \ldots,t_N) + \lambda_2 C_r(t_1, \ldots,t_N), \\
 \text{s.t.  } E(t_1, \ldots, t_N) = x_f \\
\end{array}
\end{equation*}
where $\lambda_1$ and $\lambda_2$ are two parameters, chosen to give more or less importance to the different costs. For instance, if $\lambda_1 \gg  \lambda_2$, the solution is close to the initial one $(\overline{t}_1, \ldots,\overline{t}_N)$.

\subsection{Redundancy creates robustness}

Let us consider a control $u = ((u_{10}, \ldots,u_{m0}),\left({t}_1,i_1\right), \ldots, \left({t}_N,i_N\right),t_f)$. In order to reduce the optimization space, we will consider in the following subsection that the initial control values $(u_{10}, \ldots,u_{m0})$, the components $(i_1, \ldots, i_N)$ and the final time $t_f$ are fixed, so we will forget them in the notations.

We propose here to go further in order to improve the robustness of the corresponding trajectory. We do so by adding needles to some components of the control. By needle, we mean a short impulse on one of the control. Let us denote by $l$ the number of needles we are willing to add. It means that we look for additional switching times $[(s_1,s_2), \ldots,(s_{2l-1},s_{2l})]$ and components of the control $(j_1, \ldots, j_l)$, so that for all $i \in \llbracket 1,l \rrbracket$, $(s_{2i-1},s_{2i})$ are switching times for the $j_i$-th components of the control (see Figure \ref{fig_redundancy}). It aims at giving us more degrees of freedom while trying to absorb perturbations $\delta x$ by moving the switching times $(\mathcal{T},\mathcal{S}) = (t_1, \ldots, t_N, (s_1,s_2), \ldots,(s_{2l-1},s_{2l}))$. Thus, we are solving the optimization problem
\begin{equation}
\begin{array}{ccc}
\min_{(j_1, \ldots,j_l)} 	& \min_{(\mathcal{T},\mathcal{S})} & \lambda_1 C(\mathcal{T},\mathcal{S}) + \lambda_2 C_r(\mathcal{T},\mathcal{S}). \\
				     	& \text{s.t.  } E(\mathcal{T},\mathcal{S}) = x_f \\
\end{array}
\label{prob_add_needles}
\end{equation}

\begin{remark}
If the original bang-bang control strategy $\bar{u}$ does not come from an optimization process, that is there is no cost $C$ associated with it, we can still consider problem \eqref{prob_add_needles} but with $\lambda_1 = 0$.
\end{remark}

\begin{figure}[h!]
  \hspace{0.6cm}
  \subfloat[Initial control.]{\includegraphics[scale = 0.85]{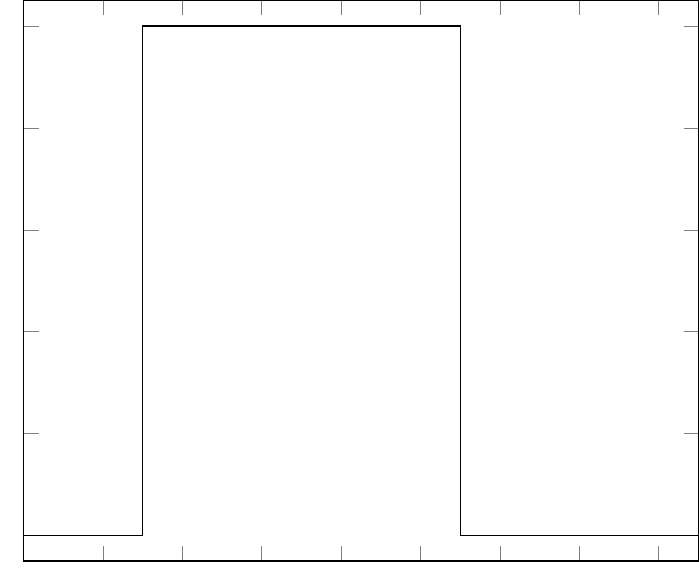}}
  \hspace{1.cm}
  \subfloat[Adding four switching times i.e., 2 needles.]{\includegraphics[scale=0.85]{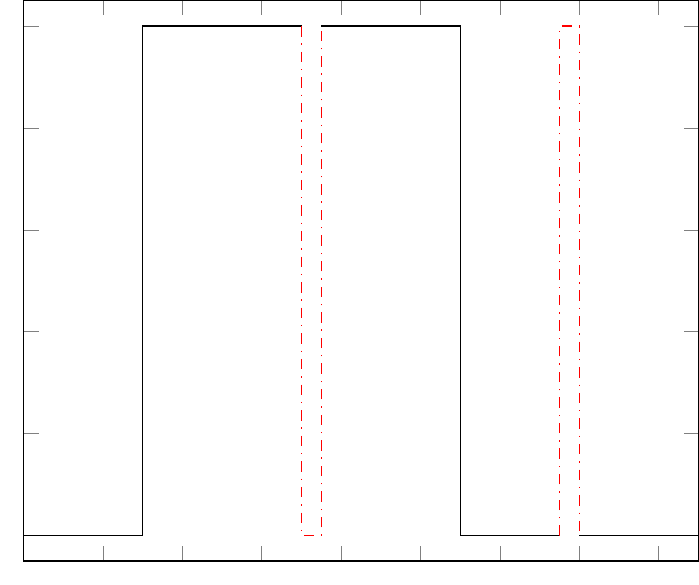}}
  \caption{Principle of adding needles.}
  \label{fig_redundancy}
\end{figure}

Let us denote by $\overline{\mathcal{T}}$ the solution of problem \eqref{prob_ini}, and by $(\mathcal{T}^{\star},\mathcal{S}^{\star})$ the solution of problem \eqref{prob_add_needles}. Then, we have that 
\begin{equation*}
C(\overline{\mathcal{T}}) \leq C(\mathcal{T}^{\star}, \mathcal{S}^{\star}).
\end{equation*} 
It means that the solution $(\mathcal{T}^{\star}, \mathcal{S}^{\star})$ is sub-optimal with respect to the initial cost $C$. However, this sub-optimality comes with a gain in terms of robustness. Besides, the loss of optimality (and therefore gain in robustness) can be controlled by the choice of the coefficients $\lambda_1$ and $\lambda_2$.

This problem is a mixed problem, with integer variables (the components $\left(j_1, \ldots, j_l\right)$), and continuous variables (the switching times $\left(t_1, \ldots, t_N, (s_1,s_2), \ldots,(s_{2l-1},s_{2l})\right)$). However, if the components are fixed, we only have to solve a non-linear problem subject to non-linear constraints in finite dimension
\begin{equation}
\begin{array}{cc}
\min_{(\mathcal{T},\mathcal{S})} & \lambda_1 C(\mathcal{T},\mathcal{S}) + \lambda_2 C_r(\mathcal{T},\mathcal{S}).  \\
 \text{s.t.  } E(\mathcal{T},\mathcal{S}) = x_f \\
\end{array}
\label{prob_fixed_modes}
\end{equation}
We used an interior-point algorithm to solve \eqref{prob_fixed_modes}. In \cite{Antsaklis2000,Zhu2015}, gradient-based algorithms are shown to be effective to solve such problems, when the sequence of indices $\left(j_1, \ldots, j_l\right)$ is fixed. 
Therefore a ``na\"ive'' way to proceed, if $m$ denotes the number of components of the control, is to solve $m^l$ optimization problems, which is extremely costly if $m$ or $l$ is big. A compromise has to be found between the potential benefit in robustness and the computational cost. Such a compromise will however depend on the particular problem at hand, so we do not elaborate too much on this issue and give an example in Section \ref{sec4}. Let us cite \cite{Chyba2008,Chyba2009}, where the authors parametrize an optimal control problem (for the time-minimal and $L^1$ problem) with the switching times of the controls. They simplify its complex structure by fixing the number of switching times, and wonder how many switching times are required to obtain a cost close to the optimal one : the result is striking as 2 or 3 may be enough. However, they know from an \emph{a priori} study the value of the optimal $L^1$ or time-minimal cost, and therefore can stop adding switching times when reaching a given percentage of this optimal value of the criterion. In our problem, we do not know what is the optimal value of the criterion we identified to quantify the robustness of a trajectory. It becomes necessary to find another way to decide how many needles to add.

One could consider tackling directly Problem \ref{prob_add_needles}, a combinatorial optimization problem (which is a class of problem known to be hard to solve). Recent years have seen the development of advanced numerical procedures to deal with the combinatorial nature of those problem at a reasonable computational cost. We give more details on this issue at the end of this section.

\begin{remark} 
Let us make here a remark on the ordering of the switching times. In the vector $(\mathcal{T},\mathcal{S})$ are stored the switching times ${t}_{i}$ and $s_i$ that represent the control ${u}$. Those swicthing times are not necessarily ordered during or after the optimization process, so let $\mathbb{T} = (\tau_1, \ldots,\tau_{N+2l})$ denote the ordered equivalent to $(\mathcal{T},\mathcal{S})$. So far, we have made the implicit assumption that when we perform the numerical integration of the system, the switching times are ordered: $\tau_{i+1} - \tau_{i} \geq 0$ for all $i \in \llbracket 0,N+2l-1 \rrbracket$. We recall that our goal is to absorb perturbations $\delta x$. As explained in Subsection \ref{backward_epm}, we compute at order one the correction to apply $\delta \mathbb{T} = d{E}(\mathbb{T})^{\dagger}\cdot \delta x$. At this point, we could have that $\mathbb{T} + \delta \mathbb{T}$ does not satisfy this ordering property. Then, we consider that $\mathbb{T} + \delta \mathbb{T}$ is not admissible, and an estimate like \eqref{estimate_deltat} would not hold.
\label{rq_gap}
\end{remark}

In the following, in order to guarantee that we do not have an interchanging of the switching times (at least for small perturbations), we add an additional constraint whilst elaborating the robustified trajectory $({u}(\cdot),{x}(\cdot))$ at \eqref{prob_add_needles}: 
\begin{equation}
\begin{array}{ccc}
{\tau}_{i+1} - {\tau}_i \geq \eta & \text{for all} & i \in \llbracket 0,N+2l-1 \rrbracket,\\
\end{array} 
\label{eq_gap}
\end{equation}
for some $\eta > 0$, where $\mathbb{T} = (\tau_1, \ldots, \tau_{N+2l})$ denotes the re-ordering of the vector $(\mathcal{T},\mathcal{S})$. In that way, we ensure that two consecutive switching times ($\mathcal{T}$ and $\mathcal{S}$ combined) are at least distant of $\eta$. Thus, if $\delta x$ is small enough, the elements of the vector $\mathbb{T} + d{E}(\mathbb{T})^{\dagger}\cdot \delta x $ remain in ascending order. Besides, such a constraint is often highly justified in practice, for instance if a physical system has to spend some minimum time $\eta$ before it switches to another mode. For example, in Section \ref{sec4}, the attitude control of a rigid body is studied. In real life, because of robustness issues and mechanical constraints, nozzles on a space launcher have indeed a minimum activation time.

\begin{remark}
Let $t_f$ denote the final time. If $\eta$ is the minimal time between two switchings in \eqref{eq_gap}, then the total number of switchings $N + 2l$ has an upper bound of $\lfloor t_f/\eta \rfloor$.
\end{remark}

The elaboration of a robust trajectory in \eqref{prob_add_needles} can be seen as an optimal control problem of switched-mode dynamical system. A recent survey on switched systems can be found in \cite{Zhu2015}. This theory deals with control systems where the dynamics can only take a finite number of modes. To determine the command law, one has to determine the switching times, as well as the different modes of the system. If the modes are fixed (in our case, it means that the components $(i_1, \ldots, i_N, j_1, \ldots, j_l)$ are fixed), it is often called a timing-optimization problem ; if not, a scheduling optimization problem. In \cite{Piccoli1999,Sussmann2000}, necessary conditions are derived, for trajectories of hybrid systems considering a fixed sequence of modes of finite length (in our setting, it corresponds to the Problem \eqref{prob_fixed_modes}). In \cite{Ali2014, Wardi2012}, the authors develop numerical algorithms to solve both the timing and the scheduling problems. Their techniques rely heavily on gradient-like methods. However, the latter problem is much more complex because of its discrete nature: indeed the procedure needs to account for both continuous and discrete control variables, and can therefore be seen as a combinatorial optimization problem. Note that the paper \cite{Ali2014} deals with dwell time constraints. It consists in imposing a threshold $\eta$ between two consecutive switching times which is the constraint we introduced at \eqref{eq_gap}. Let us also mention other techniques to solve scheduling optimization problems, like zoning algorithms \cite{Shaikh2005}, or relaxation methods, where discrete variables are temporarly relaxed into continuous variables \cite{Bengea2005}.

%%%%%%%%%%%%%%%%%%%%%%
% Section numérique
%%%%%%%%%%%%%%%%%%%%%%
\section{Numerical results}
\label{sec4}

In order to illustrate the results of Sections \ref{sec2} and \ref{sec3}, we consider the problem of the attitude control of a rigid body. Let $\omega = (\omega_1,\omega_2, \omega_3)$ be the angular velocity of the body with respect to a frame fixed on the body. Introducing the inertia matrix $I$, the Euler's equation for a rigid body, subjected to torques $(b^1, \ldots, b^m)$, writes:
\begin{equation*}
I \dot{\omega} = I \omega \wedge \omega + \sum_{k=1}^m {b^k}.
\end{equation*}

 In the case when the axes of the body frame are the axes of inertia of the body, the matrix $I$ is diagonal: $I = \diag(I_1,I_2,I_3)$. The controlled Euler's equations can then be reduced to
 \begin{equation*}
\dot{\omega}(t) = f\left(\omega(t),u(t)\right),
\end{equation*}
where for $1 \leq k \leq m$, $u_k(t) \in \{0,1\}$ almost everywhere, and the function $f$ describing the dynamics writes:
\begin{equation}
f(\omega_1,\omega_2,\omega_3,u_1,u_2,u_3,u_4) = \left\{
\begin{array}{l}
\alpha_1 \omega_2 \omega_3 + \sum^m_{k=1}{b^k_1 u_k} \\
\alpha_2 \omega_1 \omega_3 + \sum^m_{k=1}{b^k_2 u_k} ~,\\
\alpha_3 \omega_1 \omega_2 + \sum^m_{k=1}{b^k_3 u_k} \\
\end{array}\right.
\label{eq_euler}
\end{equation}
with $\alpha_1 = (I_2 - I_3)/I_1$, $\alpha_2 = (I_3 - I_1)/I_2$ and $\alpha_3 = (I_2 - I_1)/I_3$. This is with a slight abuse in the notations, because we still denote by $b^k$ the normalized vector $(b^k_1/I_1,b^k_2/I_2, b^k_3/I_3)$.

The controllability of such a system has been studied in \cite{bonnard2006mecanique}. Let us mention here the papers \cite{Krstic1999,Outbib1992,Windeknecht1963}, that implement, in the special case of the stabilization of a rigid spacecraft, methods to stabilize the spacecraft towards the point $(0,0,0)$, but once again, the controls used are not bang-bang. Note that \eqref{eq_euler} is a control-affine system, and therefore, Remark \ref{rq_affine} applies.

In the following, we consider the numerical values $\alpha_1 = 1$, $\alpha_2 = -1$, $\alpha_3 = 1$, $b^1 = [2,1,0.3]$, $b^2 = [-2,-1,-0.3]$, $b^3 = [0,0,1]$ and $b^4 = [0,0,-1]$, and initial and final conditions $x_0 = (0,0,0) $ and $x_f = (0.4,-0.3,0.4)$.

We start by building an optimal trajectory for the $L^1$ cost $\int^{t_f}_0{\sum_{j=1}^4 {|u_j(t)|dt}} + t_f$ (the presence of $t_f$ ensures us not to obtain a trajectory with infinite final time). The resolution of such a problem with a $L^1$ cost can be numerically challenging. Numerical methods in optimal control are often categorized in two categories: direct methods and indirect methods. Whereas direct methods consist in a total discretization of the state and control spaces, indirect methods exploit Pontryagin maximum principle. (see \cite{Trelat2012} for a survey on numerical methods in optimal control). The aim of the following subsection is to explain briefly the principle of a continuation method.

\subsection{Computing the nominal trajectory}

The nominal trajectory, optimal for the $L^1$ cost, is computed with a continuation procedure. The idea of such a  procedure is to solve first an ``easier'' problem, and deform it step by step to solve the targeted problem. We introduce the continuation parameter $\lambda \in [0,1]$, and we consider the optimal control problem $(\mathcal{P}_{\lambda})$ of steering the system \eqref{eq_euler} from $x_0$ to $x_f$, by minimizing the cost 
\begin{equation*}
\lambda \int^{t_f}_0{\sum^4_{i=1}{|u_j(t)|^2\, dt}} + (1-\lambda) \int^{t_f}_0{\sum^4_{i=1}{|u_j(t)|\, dt}} + t_f.
\end{equation*}
When $\lambda = 0$, we recognize our problem.
For some $\lambda \in [0,1]$, solving problem $(\mathcal{P}_{\lambda})$ is done by finding the zeros of a shooting function that results from the application of Pontryagin maximum principle. Solving a shooting problem is done with Newton like methods. Such methods are highly sensitive to their initialization, that can be very difficult, especially in the case of the minimization of the $L^1$ norm $\int^{t_f}_0{|u(t)|dt}$. The continuation procedure is introduced to overcome this difficulty.

For $\lambda=1$, the cost is stricly convex in the controls, and writes 
\[
\int^{t_f}_0{\sum^4_{i=1}{|u_j(t)|^2\, dt}} + t_f,
\]
for which the initialization of the induced shooting method is much easier. Therefore, we solve a sequence of optimal control problems, for values of $\lambda$ decreasing from 1 to 0. The result of the shooting problem for some $\lambda \in ]0,1]$ serves as the initialization of another problem with $\lambda' < \lambda$.

\subsection{Robustifying the nominal trajectory}

From this $L^1$ - minimal trajectory, represented on Figure \ref{fig_add_needles}, with three switching times that we denote $(t_1,t_2,t_3)$ we build a new trajectory by solving the problem \eqref{prob_add_needles} with 3 needles (i.e., $l=3$), $\lambda_1 = \lambda_2 = 1$, and taking $\eta = 0.05$ in Equation \eqref{eq_gap}. As explained in Remark \ref{rq_leastsquares}, we see that it is worthwile to have the additional switching times available as long as possible. That is, we force the additional switchings to occur after $t_3$. Keeping in mind Equation \eqref{eq_gap}, this constraint can be written:
\[
\begin{array}{ccc}
t_{i+1} - t_i \geq \eta~~ (\forall i \in \llbracket1,3\rrbracket) ,& s_{1}-t_{3} \geq \eta ,& s_{i+1} - s_i \geq \eta~~ (\forall i \in \llbracket1,6\rrbracket).
\end{array}
\]
We find that the optimal triplet is $(j_1,j_2,j_3) = (1,4,2)$, for which we have $C = 0.77$ and $C_{r} = 2.22$. We found this optimal triplet by exploring the $4^3 = 64$ possibilities. We then used the heuristic that this solution would make a good choice to start looking for the solution with 4 needles (as it would have been to costly to examine the $4^4 = 256$ possibilities). However we could not make the cost dicrease significantly (the best cost we found was $C_r = 2.07$). This heuristic is very similar to what is used in Branch and Bound methods. Besides, as an element of comparison, the optimal couple when adding only two needles is $(j_1,j_2) = (1,4)$, for which $C_r = 4.25$, and the optimal solution when adding only on needle is $j_1 = 2$, for which $C_r = 30.28$. Thus, we notice a substantial improvement when increasing the number of needles from 1 to 2 and from 2 to 3, whereas it seems less profitable to add a fourth one. We therefore stopped at 3 needles.
The controls are displayed on Figure \ref{fig_add_needles}, and the components 1, 2 and 4, on which needles have been added, are represented in red.

\begin{figure}[h!]
\begin{center}
  \subfloat[Controls for the minimal $L^1$ trajectory.\ $C =  0.49$]{\includegraphics[scale=0.80]{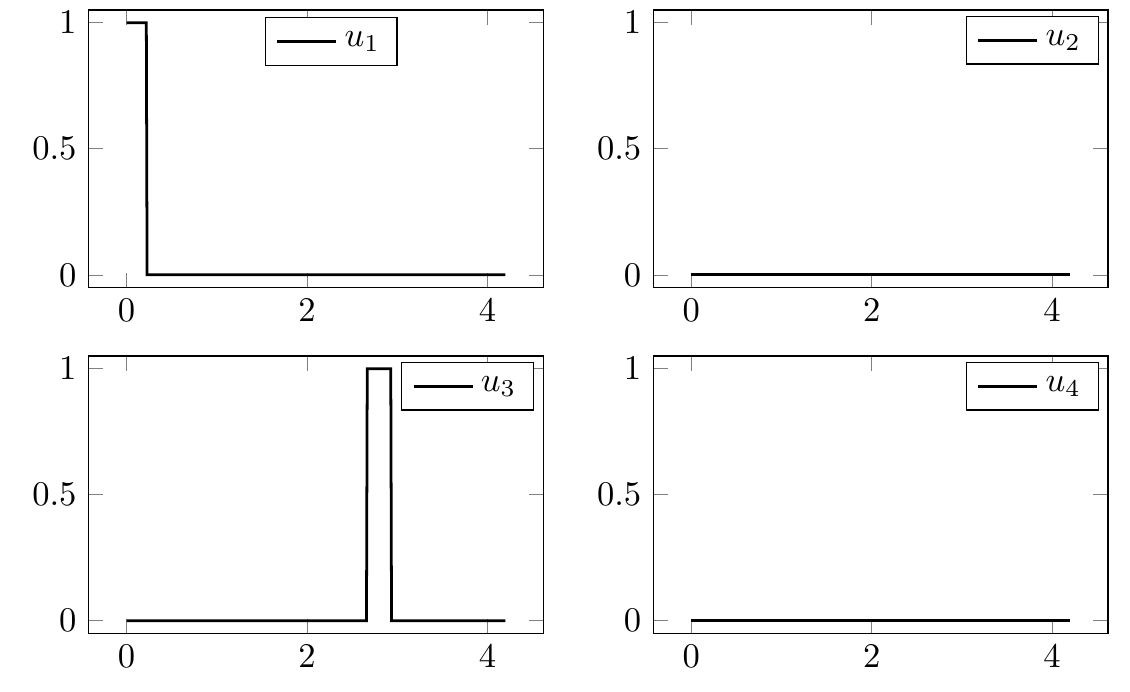}}
 \end{center}

  \begin{center}
  \subfloat[Controls with three needles. $C = 0.77$, $C_r = 2.22$]{\includegraphics[scale=0.80]{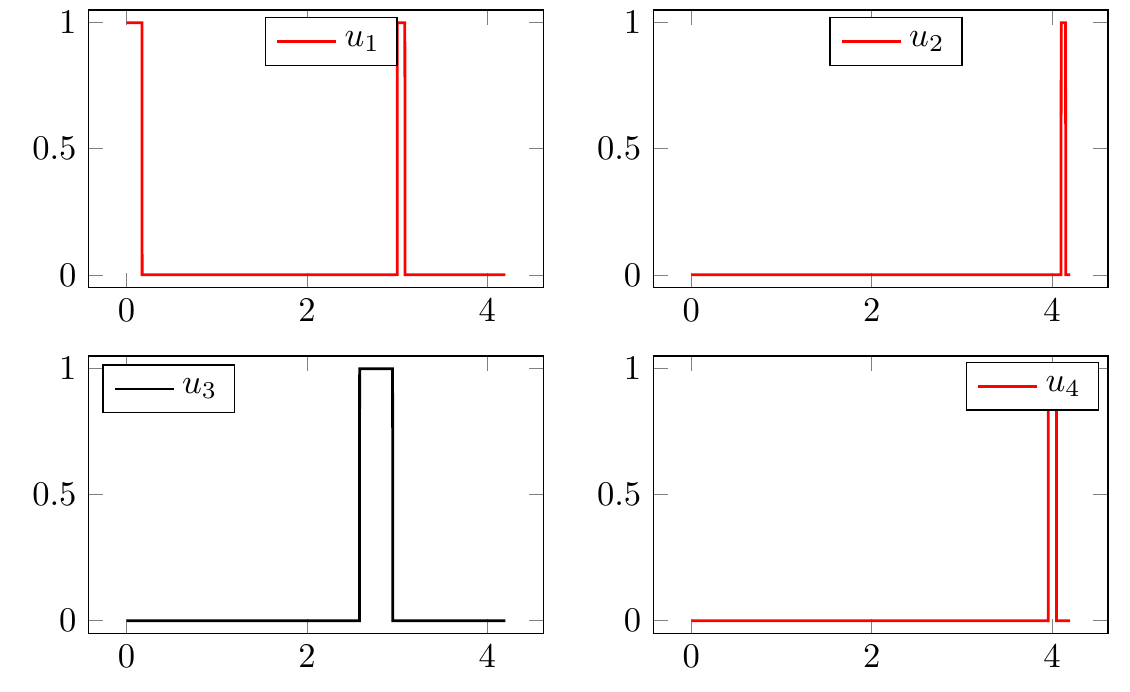}}
  \end{center}
  \caption{Improving the robustness of a trajectory adding needles.}
  \label{fig_add_needles}
\end{figure}

In order to represent perturbations, we consider that the principal moments of inertia can vary, causing the coefficients $\alpha_1$, $\alpha_2$ and $\alpha_3$ to vary. Thus we consider the perturbed dynamics
\begin{equation}
f_{per}(t,\omega_1,\omega_2,\omega_3,u_1,u_2,u_3,u_4) = \left\{
\begin{array}{l}
\alpha^{per,\varepsilon}_1(t) \omega_2 \omega_3 + \sum^m_{k=1}{b^k_1 u_k} \\
\alpha^{per,\varepsilon}_2(t) \omega_1 \omega_3 + \sum^m_{k=1}{b^k_2 u_k}~, \\
\alpha^{per,\varepsilon}_3(t) \omega_1 \omega_2 + \sum^m_{k=1}{b^k_3 u_k} \\
\end{array}\right.
\label{eq_euler_per}
\end{equation} 
so that $\varepsilon$ models the size of the perturbation. More precisely, we take $\alpha_i^{per,\varepsilon}(t) = \alpha_i + \varepsilon h_i(t)$, where $h_i(\cdot)$ is some periodic function satisfying $\left\| h_i \right\|_{\infty} \leq 1$ (note that the exact expression of $h_i$ is not relevant here, as it is supposed to model any perturbation of the $\alpha_i$). We denote by $x_{per}$ the solution of the Cauchy problem 
\begin{align*}
\dot{x}(t) &= f_{per}(t,x(t),u(t)), \\
x(0) &= x_0.
\end{align*}

We denote by $x_{cor}$ the corrected trajectory computed with our algorithm.
We show, on Figure \ref{fig_tracking}, the three trajectories, for $\varepsilon = 0.78$ and a cost $C_{r} = 2.22$. We can see the perturbed trajectory $x_{per}$ drifting away from the reference trajectory $x_{ref}$ and away from the final point $x_f$, whereas the corrected trajectory $x_{cor}$ eventually reaches a point very close to $x_f$. Actually, for the trajectories represented on Figure \ref{fig_tracking}, we have that $\left\| x_{cor}(t_f) - x_f \right\| / \left\|x_f\right\| = 5.5 \times 10^{-3}$, whereas $\left\| x_{per}(t_f) - x_f \right\| / \left\|x_f\right\| = 1.3 \times 10^{-1}$. Our algorithm has indeed been able to adjust the perturbed trajectory back towards $x_f$.

One may wonder how this method behaves with respect to the choice of $\varepsilon$. As explained in Remark \ref{rq_interchanging}, we stop if two switching times are interchanged, that is, if $\delta T$ is too big, as the initial vector of switching times satisfies a gap property \eqref{eq_gap}. Actually, this is not strictly true, as we could have a ``big'' correction that does not change the ascending order of the switching times, for instance if we shift all the switching times in the same direction. However, we experimentally notice that the cost $C_{r}$ \emph{has an impact on the size of the perturbation we are able to absorb}. 

We build several trajectories, for which we apply our algorithm for increasing values of $\varepsilon$, until the algorithm fails as explained in Remark \ref{rq_interchanging}, for some $\varepsilon_{\max}$. We plot on Figure \ref{fig_size_perturbations} the value of $\varepsilon_{\max}$ with respect to the cost $C_{r}$ (that is, for a given cost $C_r$, $\varepsilon_{max}$ is the smallest value for which there is an interchanging of switching times). Even if the curve is not decreasing (for the reason explained above), we can see that \emph{having a low cost $C_{r}$ enables us to absorb bigger perturbations}.

\begin{figure}[!h]
\begin{center}
\includegraphics{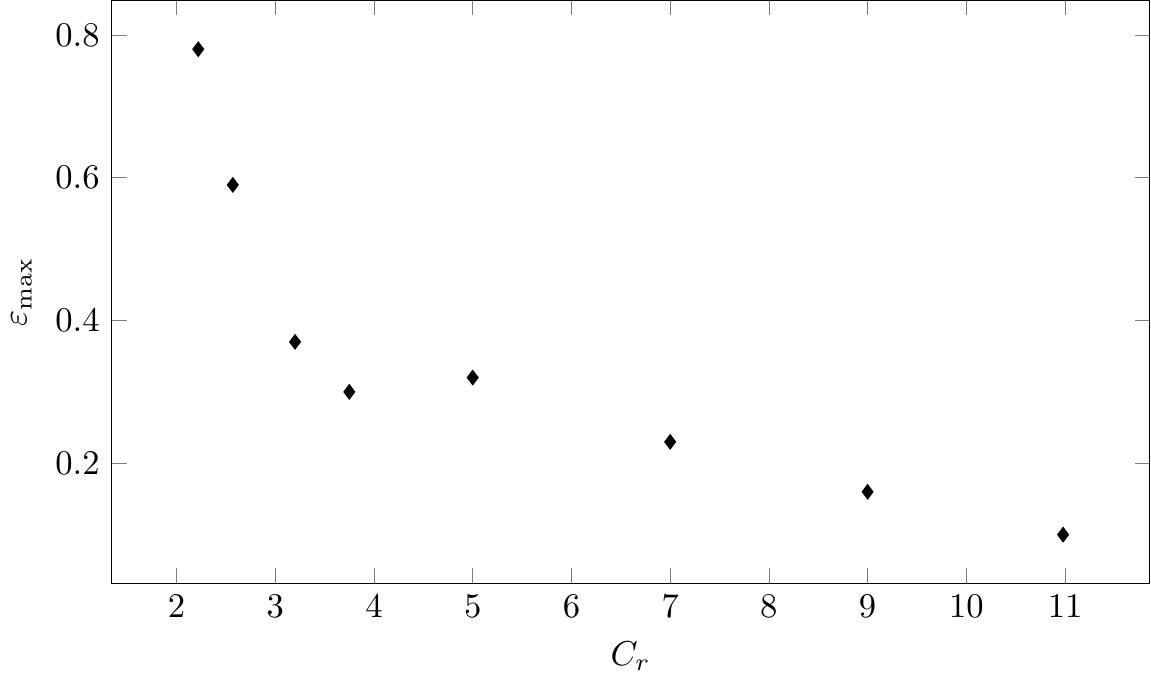}
\caption{Size of the maximal perturbation absorbed with respect to the robustness of a trajectory}
\label{fig_size_perturbations}
\end{center}
\end{figure}

\begin{figure}
\vspace{-1cm}
\includegraphics{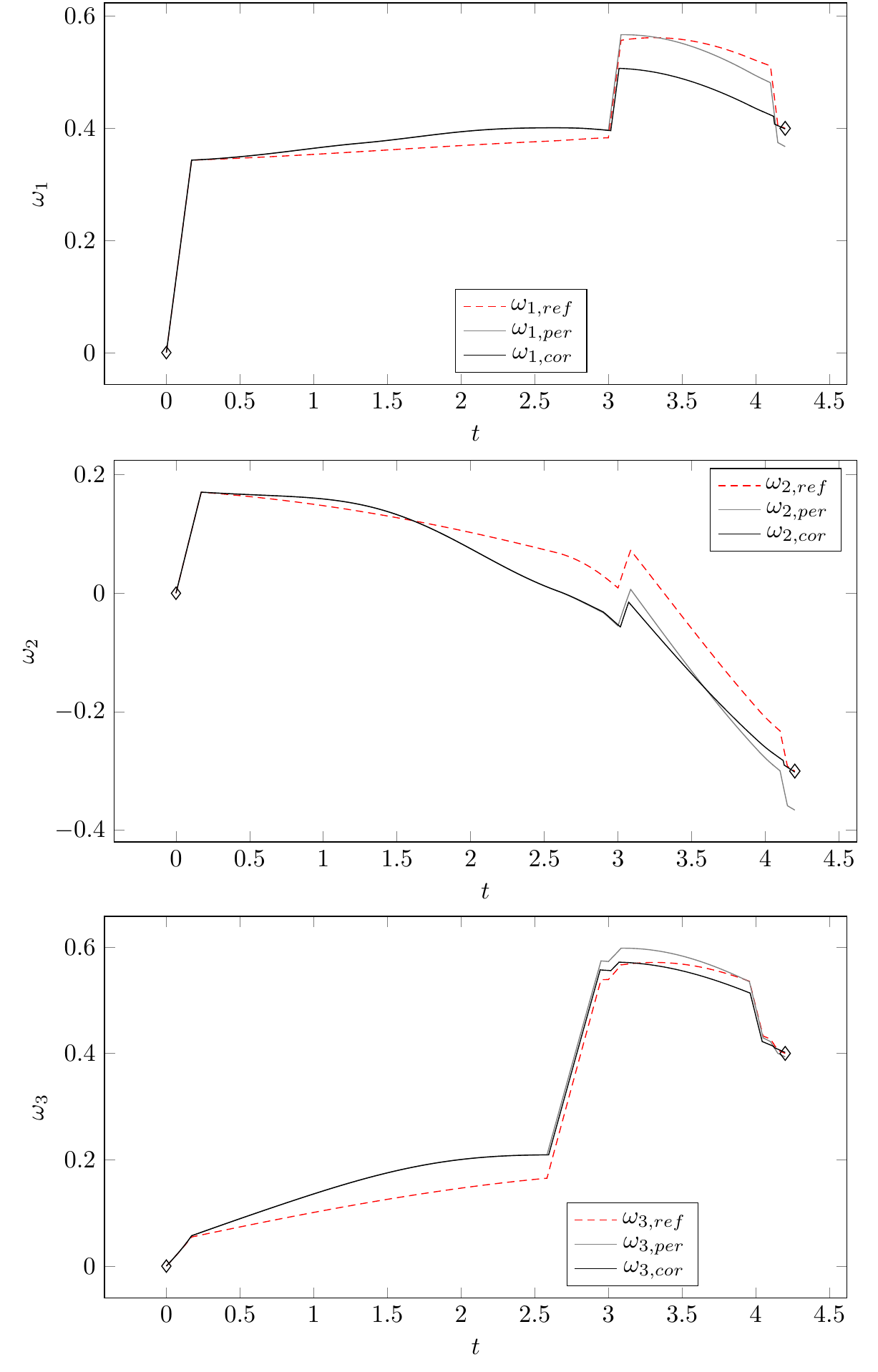}
\caption{Reference, perturbed and corrected trajectories for $\varepsilon = 0.78$, $C_{r} = 2.22$.}
\label{fig_tracking}
\end{figure}

On Figure \ref{fig_tracking_results}, we show the relative error $\|x(t_f)-x_f\|/\|x_f\|$ for the perturbed $x_{per}$ and corrected $x_{cor}$ trajectories, for several values of $\varepsilon$. As we apply order one corrections, we see that our method shows better results for small values of $\varepsilon$, but also gives very satisfactory results for larger values of $\varepsilon$.

\begin{figure}
\vspace{-1cm}
\includegraphics{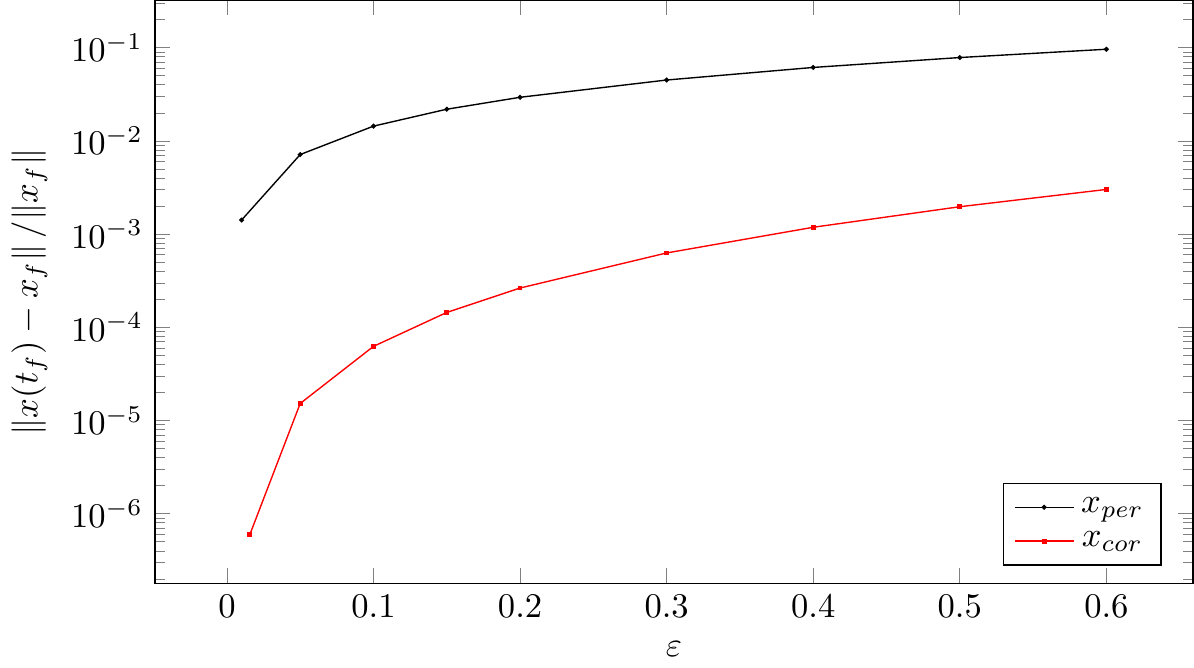}
\caption{Tracking results for several values of $\varepsilon$.}
\label{fig_tracking_results}
\end{figure}

\section{Conclusion}

Starting with the expansion of the end-point mapping with respect to a needle like variation, we have shown in this paper how redundant switching times can be added in order to make a control more robust, for general control systems of the form $\dot{x}(t) = f(t,x(t),u(t))$. Those additional switching times can be seen as extra degrees of freedom meant to help us absorb perturbations. A potential application is to start from a bang-bang solution of an optimal control problem, that is usually not robust, and make it more robust. Then the gain in robustness compensates for the loss in optimality.

In the presence of a perturbation $\delta x$, the correction to apply to the switching times is the solution of an equation $dE \cdot \delta \mathcal{T} = \delta x$. It is natural to try to solve this equation while shifting the switching times as little as possible. The least-squares problem formulation is then the appropriate setting to find the solution of minimal (euclidian) norm of the previous equation, and it is given by $\delta \mathcal{T} = dE^{\dagger} \cdot \delta x$, for which we have the norm estimation $\left\| \delta \mathcal{T} \right\|_2 \leq \left\| \delta x \right\|_2/\sigma_{min}$. This enabled us to identify the measure for robustness:
\begin{equation*}
\int{\frac{1}{\sigma_{min}(t)^2}\, dt}.
\end{equation*}

The numerical example studied in Section \ref{sec4} is academic, and was used to legitimize the theoretical ideas explained previously. In a future work, we aim at applying the method to the complete (and more complex) attitude control system of a three-dimensional rigid body, for which we wish to control the angular velocity, as well as the orientation with respect to a fixed reference frame. To the three velocity variables will be added three angles to parametrize the orientation of the body. Thus, a challenge will come from the dimension of the state space (6), as well as the potentially bigger number of needle-like variations required to robustify a trajectory.

%%%%%%%%%%%%%%%%%%%%%%%%%%%
% Annexe
%%%%%%%%%%%%%%%%%%%%%%%%%%%
\appendix
\section{Proof of proposition \ref{prop_differentiability_epm}}
  
In order to prove the differentiability of the end-point mapping, we start with the differentiability with respect to one component. The proof relies heavily on the expansion \eqref{eq_var}, that we recall first.

\begin{lemma}
\label{lemme_aiguilles}
Let $t_1 \in [0,t_f[$, and let $u_{\pi_1}(\cdot)$ be a needle-like variation of $u(\cdot)$, with $\pi_1=(t_1,\delta t_1,u_1)$. Then
\begin{equation*}
x_{\pi_1}(t_f) =  {x}(t_f) + \left| \delta t_1 \right| v_{\pi_1}(t_f) + \mathrm{o}(\delta t_1),
\end{equation*}
where $v_{\pi_1}(\cdot)$ is the solution of a Cauchy problem on $[t_1,t_f]$ 
\begin{align*}
\dot{v}_{\pi_1}(t) & = \frac{\partial f}{\partial x}(t, {x}(t), {u}(t))v_{\pi_1}(t),\\
v_{\pi_1}(t_1) & =  f(t_1, {x}(t_1),u_1) - f(t_1, {x}(t_1), {u}(t_1)).
\end{align*}
\end{lemma}

\begin{proposition}
We denote by $u$ the control $(t_1, \ldots, t_N,t_f)$ and $x(\cdot)$ the associated trajectory of the control system. Let $\delta t_1 \in \R$ be small enough. Then
\begin{equation*}
E(t_1+\delta t_1, t_2, \ldots, t_N, t_f) = E(t_1, \ldots, t_N, t_f) + \delta t_1 \cdot v_1(t_f) + o(\delta t_1),
\end{equation*}
where $v_1(\cdot)$ is the solution of the Cauchy problem on $[t_1,t_f]$:
\begin{align*}
\dot{v}_1(t) &= \frac{\partial f}{\partial x}(t,x(t),u(t))v_i(t), \\
v_1(t_1) &= \left\{
\begin{tabular}{rl}
$f(t_1,x(t_1),(\ldots, a_{i_1}, \ldots)) - f(t_1,x(t_1),u(t_1^+))$ & if $u_{i_1}$ switches from $a_{i_1}$ to $b_{i_1}$. \\
$f(t_1,x(t_1),(\ldots, b_{i_1}, \ldots)) - f(t_1,x(t_1),u(t_1^-))$ &  if $u_{i_1}$ switches from $b_{i_1}$ to $a_{i_1}$.
\end{tabular}
\right.
\end{align*}
\end{proposition}

\begin{proof}

\begin{figure}[h!]
  \hspace{0.3cm}
  \subfloat[$\delta t_1 > 0$.]{\includegraphics[scale = 0.9]{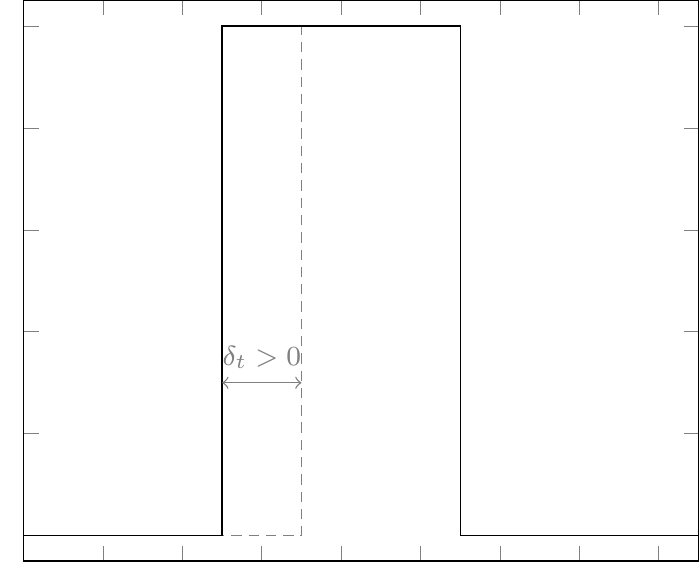}}
  \hspace{1.cm}
  \subfloat[$\delta t_1 < 0$.]{\includegraphics[scale=0.9]{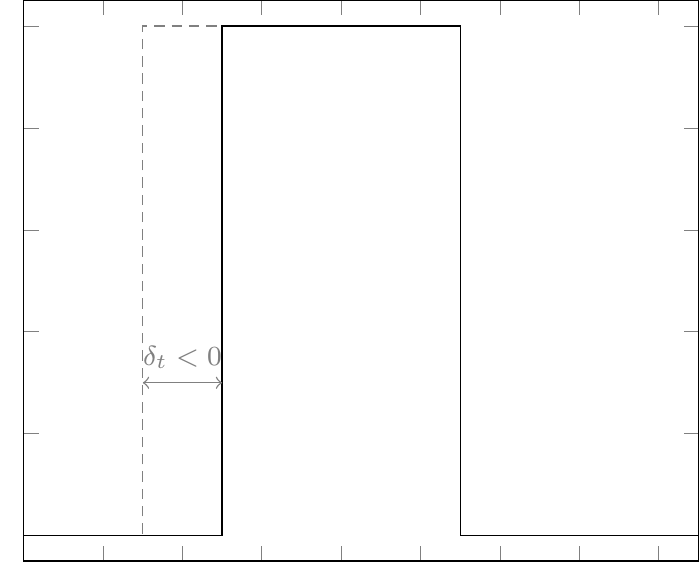}}
  \caption{Shifting an opening time is equivalent to add a needle.}
  \label{aiguille_demo}
\end{figure}
Assume that at time $t_1$ the control $u_{i_1}$ switches from $a_{i_1}$ to $b_{i_1}$, and that $\delta t_1 > 0$. Let us define the needle-like variation $\pi = (t_1,\delta t_1, a_{i_1})$ for the $i_1$-th component of the control. Then, the control $u_{\pi}$ is represented by the vector $(t_1+\delta t_1, \ldots, t_N, t_f)$ (figure \ref{aiguille_demo}): adding the needle-like variation $\pi$ to the $i_1$-th component, with value $a_{i_1}$ and length $\delta t_1$ is equivalent to shifting the opening time to $t_1+\delta t_1$. Thus, we have that $u(t_1^+)_{i_1} = b_{i_1}$ and $u_{\pi}(t_1^+)_{i_1} = a_{i_1}$. Hence, we obtain that, according to lemma \ref{lemme_aiguilles} 
\begin{equation}
\label{var_pos}
x_{\pi}(t_f) = x(t_f) + \delta t_1 \cdot v_1(t_f) + o(\delta t_1),
\end{equation}
where $v_1(\cdot)$ is the solution of the Cauchy problem:
\begin{align*}
\dot{v}_1(t) &= \frac{\partial f}{\partial x}(t,x(t),u(t))v_1(t), \\
v_1(t_1) 	&= f(t_1,x(t_1),u_{\pi}(t_1^+)) - f(t_1,x(t_1),u(t_1^+)) \\
		&= f(t_1,x(t_1), (\ldots, a_{i_1}, \ldots)) - f(t_1,x(t_1), (\ldots, b_{i_1}, \ldots)).
\end{align*}
(Between $u_{\pi}(t_1^+)$ and $u(t_1^+)$, only the $i_1$-th component differs.)

If $\delta t_1 <0$, define the variation $\pi = (t_1,\delta t_1, 1)$ for the $i_1$-th component of the control. Then again, the control $u_{\pi}$ is represented by the vector $(t_1+\delta t_1, \ldots, t_N, t_f)$ (figure \ref{aiguille_demo}). Thus, we have that $u(t_1^-)_j = a_{i_1}$ and $u_{\pi}(t_1^-)_{i_1} = 1$. Thanks to lemma \ref{lemme_aiguilles}, we obtain that 
\begin{equation}
\label{var_neg}
x_{\pi}(t_f) = x(t_f)  - \delta t_1 \cdot w_1(t_f) + o(\delta t_1),
\end{equation}
where $w_1(\cdot)$ is the solution of the Cauchy problem:
\begin{align*}
\dot{w}_1(t) 	&= \frac{\partial f}{\partial x}(t,x(t),u(t))w_1(t), \\
w_1(t_1) 		&= f(t_1,x(t_1),u_{\pi}(t_1^-)) - f(t_1,x(t_1),u(t_1^-)) \\
			&= f(t_1,x(t_1), (\ldots, b_{i_1}, \ldots)) - f(t_1,x(t_1), (\ldots, a_{i_1}, \ldots))  \\
			&= -v_1(t_1).
\end{align*}
Thus, by uniqueness we have $w_1 = -v_1$, and from \eqref{var_pos} and \eqref{var_neg}, we obtain:
\begin{equation*}
x_{\pi}(t_f) = x(t_f) + \delta t_1 \cdot v_1(t_f) + o(\delta t_1).
\end{equation*}

We can proceed the exact same way if  at $t_1$, the control $u_{i_1}$ switches from $b_{i_1}$ to $a_{i_1}$
\end{proof}

The general result at proposition \ref{prop_differentiability_epm} follows by an immediate iteration.

\renewcommand{\abstractname}{Acknowledgements}
\begin{abstract}
 This study has been performed in the frame of the CNES Launchers Research \& Technology program.
\end{abstract}

\bibliographystyle{unsrt}
\bibliography{biblio_cra.bib}

\end{document}